\documentclass{amsart}
\usepackage{amssymb,latexsym,amscd}
\input xy 
\xyoption{all}
\numberwithin{equation}{section}

\newtheorem{theorem}{Theorem}[section]
\newtheorem{proposition}[theorem]{Proposition}
\newtheorem{conjecture}[theorem]{Conjecture}

\newtheorem{corollary}[theorem]{Corollary}
\newtheorem{remark}[theorem]{Remark}
\newtheorem{lemma}[theorem]{Lemma}
\newtheorem{example}[theorem]{Example}
\newtheorem{definition}[theorem]{Definition}

\def\AA{\mathfrak{A}}

\def\ZZ{\mathbb{Z}}
\def\QQ{\mathbb{Q}}
\def\CC{\mathbb{C}}

\def\UU{\mathcal{U}}
\def\II{\mathcal{I}}

\def\1{\mathbb{1}}

\def\mm{\mathfrak{m}}

\begin{document}

\title{Poisson  and quantum geometry of acyclic cluster algebras}
\author{Sebastian Zwicknagl}
\maketitle

\begin{abstract}
   We prove that certain acyclic cluster algebras over the complex numbers are the coordinate rings of holomorphic symplectic manifolds. We also show that the corresponding quantum cluster algebras have no non-trivial prime ideals. This allows us to give evidence for a generalization of the conjectured variant of the orbit method for quantized coordinate rings and their classical limits. 
\end{abstract}

\tableofcontents
\section{Introduction}

In the present paper we investigate the Poisson geometry associated with cluster algebras over the complex numbers defined by acyclic quivers, and relate them to the ideal theory of the corresponding quantum cluster algebras. Our main motivation is the  following conjectural  analogue of Kirillov's Orbit method for quantized coordinate rings which has been an open problem for roughly twenty years (see e.g.~\cite{brown-goodearl} or \cite[Section 4.3]{Yak-spec} and \cite{Soi} for the case of compact quantum groups). Let $G$ be a semisimple complex algebraic group and $\CC[g]$ its coordinate ring  while $\CC_q[G]$ denote the corresponding quantized coordinate ring. It has been conjectured that there exists a  homeomorphism between the space of primitive ideals in $R_q[G]$ and the symplectic leaves of the standard Poisson structure on $G$. For an excellent introduction to this conjecture we refer the reader to Goodearl's paper \cite{Goo1}. The conjecture appears extremely difficult to prove and it is only known to be true in the cases of $G=SL_2, SL_3$. 

The coordinate rings $\CC[G]$ are known to have an upper cluster algebra structure (\cite{BFZ}) while the quantized coordinate rings are conjectured to have a quantum (upper) cluster algebra structure (\cite[Conjecture 10.10]{bz-qclust}). Indeed, it follows from recent results of  Gei\ss,  Leclerc and Schr\"oer (\cite[Section 12.4]{GLSq}) that  $\CC_q[SL_n]$ has a quantum cluster algebra structure.   Cluster algebras  are nowadays very well-established, hence we do not recall any of the definitions here, and refer the reader to the literature, resp.~our Section \ref{se:Cluster Algebras}. Most importantly for our purposes, a cluster algebra over $\CC$ is defined by a combinatorial datum in a field of fractions $\CC(x_1,\ldots, x_n)$.   We will denote this {\it initial seed} by $({\bf x}, B)$ where ${\bf x}=(x_1,\ldots,x_n)$ and $B$ is an integer  $m\times n$-matrix with $m\le n$ such that its principal $m\times m$ submatrix is skew-symmetrizable.  The cluster variables $x_{m+1},\ldots, x_n$ are the frozen variables which we will call coefficients.   A quantum cluster algebra is given by a {\it quantum seed} $( {\bf x}, B,\Lambda)$ where 
$(B,\Lambda)$ where $B$ is as above and $\Lambda$ is a skew-symmetric $n\times n$-matrix such that $(B,\Lambda)$ is a {\it compatible pair} (see Section \ref{se:Cluster Algebras} for details). The set ${\bf x}=(x_1,\ldots, x_n)$ now lives in the skew-field of fractions $\CC_\Lambda(x_1,\ldots,x_n)$ defined by $\Lambda$.   A compatible pair also defines a  compatible Poisson structure in the sense of \cite{GSV} on the cluster algebra given  by $({\bf x},B)$.  It is well-known  that the conjectured quantum cluster algebra structures on the rings $R_q[G]$ and the standard Poisson structure on $\CC[G]$ arise from such a compatible pair. Therefore,  we would like to suggest the following conjecture..

\begin{conjecture}
\label{conj: homeo cluster}
 Let $(B,\Lambda)$ be a compatible pair and let $\AA$ and $\AA_q$ be a cluster, resp.~quantum cluster algebra defined by $({\bf x},B)$, resp.~$({\bf x},B,\Lambda)$. Suppose further that $\AA$ and $\AA_q$ are Noetherian and that $\AA$ is the coordinate ring of the affine variety $X$. Then, there exists a homeomorphism between the space of primitive ideals of $\AA_q$ and the symplectic leaves on $X$ defined by $\Lambda$.
\end{conjecture}

In light of Conjecture \ref{conj: homeo cluster}, we may think of quantum affine space and quantum tori as cluster algebras where all cluster variables are frozen. In this case the corresponding homeomorphism is well known and easy to construct. The other extreme case are cluster algebras without coefficients and here the class that is usually easiest to study are the acyclic cluster algebras. For example, it is known that such a cluster algebra is always Noetherian and the coordinate ring of an affine variety (see \cite{BFZ} and \cite{bz-qclust} for the classical and  quantum versions). It is our main objective to give evidence for Conjecture \ref{conj: homeo cluster} by proving  it in this very specific case. It is an immediate consequence of the following two main results.  
\begin{theorem}
\label{th:Classical-intro} 
Let  $\AA$ be a cluster algebra with initial seed $({\bf x}, B)$ defined by an acyclic quiver where $B$  is invertible satisfying \eqref{eq:poisson gen}, and  suppose that it is  the coordinate ring  of an affine variety $X$ and that $(B,\Lambda)$ is a compatible pairs. Then $X$ has the structure of a symplectic manifold, whose symplectic form is the corresponding Poisson bivector. 
\end{theorem}

\begin{theorem}
\label{th:quantum-intro}
Let $\AA_q$ be a quantum cluster algebra with quantum seed  $( {\bf x}, B,\Lambda)$ satisfying the assumptions of Theorem \ref{th:Classical-intro}. Then $\{0\}$ is the only proper two sided prime ideal in $\AA_q$.
\end{theorem}

Our approach, is similar to that of \cite{ZW tpc}, however all the proofs are self-contained and much easier, as our set-up is less general. The main idea is to study the intersection of ideals with the polynomial ring generated by a given cluster--in this case the acyclic seed.  We are able to derive rather strong conditions that Poisson  prime ideals--resp.~two-sided prime ideals in the quantum case-- must satisfy and are   able to show that no non-trivial   ideals satisfying them  exist. A straightforward argument, then allows us to conclude that the variety $X$ is a symplectic manifold. We should also remark that we do not know whether any acyclic cluster algebras exist that do not satisfy the assumptions made in Theorem \ref{th:Classical-intro}.

The paper is organized as follows. We first briefly recall the definitions of cluster algebras and compatible Poisson structures, compatible pairs and quantum cluster algebras  in Section \ref{se:Cluster Algebras}. Thereafter, we continue with some technical key propositions (Section \ref{se:key propositions}) and  discuss  in  Section \ref{se:Poisson and symp geom}    the symplectic geometry of acyclic cluster algebras. The proof of Conjecture \ref{conj: homeo cluster} in our specific case  is completed in Section   \ref{se:ideals in qca} by proving Theorem \ref{th:quantum-intro}.

\section{Cluster Algebras}
\label{se:Cluster Algebras}
\subsection{Cluster algebras}
 \label{se:Cluster Algebras-def}
 In this section, we will review the definitions and some basic results on
cluster algebras, or more precisely, on cluster algebras of geometric type over the field of complex numbers $\CC$. Denote by $\mathfrak{F}=\CC(x_1,\ldots,
x_n)$ the field of fractions  in $n$ indeterminates.  Recall that a $m
\times
m$-integer matrix $B'$ is called skew-symmetrizable if
there exists a $m \times m$-diagonal matrix  $D$ with positive integer entries  
such that  $B' \cdot D$ is skew-symmetric.
 Now, let $B$ be a $m\times n$-integer matrix such that its principal $m\times m$-submatrix is skew-symmetrizable. We call the tuple $(x_1,\ldots,x_n, B)$ the {\it initial seed} of the cluster
algebra and   $ (x_1,\ldots x_m)$ a cluster, while ${\bf x}=(x_1,\ldots x_n)$ is
called an extended cluster. The cluster variables $x_{m+1},\ldots,x_n$ are called {\it coefficients}. We will now construct more  clusters, $(y_1,\ldots,
y_m)$ and extended clusters ${\bf
y}=(y_1,\ldots, y_n)$, which are transcendence bases of $\mathfrak{F}$, and the
corresponding 
seeds $({\bf y}, \tilde B)$ in the following way.

Define for each real number $r$ the numbers $r^+={\rm max}(r,0)$ and  $r^-={\rm
min}(r,0)$.
Given a skew-symmetrizable integer $m \times n$-matrix $B$, we define for each
$1\le i\le m$
the {\it exchange polynomial}
\begin{equation}
 P_i = \prod_{k=1}^n x_k^{b_{ik}^+}+ \prod_{k=1}^n  x_k^{-b_{ik}^-}\ .
\end{equation}

We can now define the new cluster variable $x_i'\in\mathfrak{F}$ via the equation
\begin{equation}
\label{eq:exchange}
 x_ix_i'=P_i\ . 
\end{equation}

This allows us to refer to the matrix $B$ as the {\it exchange matrix} of the
cluster $(x_1,\ldots,x_n)$, and to the relations  defined by Equation \ref{eq:exchange} for $i=1,\ldots,m$ as {\it exchange relations}. 

We obtain that $(x_1,x_2,\ldots, \hat x_i,x_i',x_{i+1},\ldots, x_n)$ is a
transcendence basis of $\mathfrak{F}$. We next construct  the new exchange matrix
$B_i=B'=(b_{ij}')$, associated to the new (extended) cluster $${\bf x}_i=(x_1,x_2,\ldots, \hat
x_i,x_i',x_{i+1},\ldots, x_n)$$
via its coefficients $b_{ij}'$ as follows: 

$\bullet$ $b_{ij}' = -b_{ij}$ if $j \le n$ and $i = k$ or $j = k$,

$\bullet$ $b_{ij}' =  b_{ij} + \frac{|b_{ik} |b_{kj} + b_{ik} |b_{kj} |}{2}$ if
$j \le n$ and $i \ne k$ and $j \ne k$,

$\bullet$ $b_{ij}'=b_{ij}$ otherwise.

This algorithm is called  {\it matrix mutation}. Note that $B_i$ is again
skew-symmetrizable (see e.g.~\cite{FZI}). The process of obtaining a new seed is
called {\it cluster mutation}. The set of seeds obtained from a given seed $({\bf x},B)$ is  called the mutation equivalence class of  $({\bf x},B)$.

\begin{definition}
The cluster algebra $\mathfrak{A}\subset \mathfrak{F}$ corresponding to an
initial seed $(x_1,\ldots, x_n,B)$ is the subalgebra of $\mathfrak{F}$,
generated by the elements of all the clusters in the mutation equivalence class of $({\bf x},B)$ . We refer to the elements of the clusters as the {\it cluster variables}.
\end{definition}

\begin{remark}
Notice that  the coefficients, resp.~frozen variables $x_{m+1},\ldots, x_n$  
will never be mutated. Of course, that explains their name.
\end{remark}

We have the following fact, motivating the definition of cluster algebras in the
study of total positivity phenomena and canonical bases.

\begin{proposition} \cite[Section 3]{FZI}(Laurent phenomenon) Let $\mathfrak{A}$
be a cluster algebra with
initial extended cluster $(x_1,\ldots, x_n)$. Any cluster variable $x$ can be
expressed uniquely as a Laurent polynomial in the variables
$x_1,\ldots, x_n$ with integer coefficients.  
\end{proposition}

Moreover, it has been conjectured for all cluster algebras, and proven in many cases (see
e.g.~\cite{MSW} and \cite{FST},\cite{FT})  that the coefficients of these polynomials are
positive.

Finally, we recall the definition of the lower bound of a cluster algebra $\AA$
corresponding to a seed $({\bf x}, B)$. Denote by $y_i$ for $1\le i\le m$ the
cluster variables obtained from ${\bf x}$ through mutation at $i$; i.e., they
satisfy the relation $x_iy_i=P_i$.
 \begin{definition}\cite[Definition 1.10]{BFZ}
\label{def:lower bounds}
Let $\AA$ be a cluster algebra and let $({\bf x}, B)$  be a seed. 
The lower bound $ \mathfrak{L}_B \subset \AA$ associated with $({\bf x}, B)$ is the algebra
 generated by the set  $\{x_1,\ldots x_n,y_1\ldots, y_m\}$.
\end{definition}

\subsection{Upper cluster algebras}
\label{se:upper cluster algebras}
 Berenstein, Fomin and Zelevinsky
introduced the related concept of upper cluster algebras in \cite{BFZ}.  

\begin{definition}
  Let $\mathfrak{A} \subset \mathfrak{F}$ be a
  cluster algebra with initial cluster $(x_1, \ldots, x_n, B)$ and let, as above, $y_1,
  \ldots, y_m$ be the cluster variables obtained by mutation in the directions
  $1, \ldots, m$, respectively.
  
  \noindent(a) The upper bound $\UU_{{\bf x},B} ( \mathfrak{A})$ is defined as
  \begin{equation} 
\UU_{{\bf x},B} ( \mathfrak{A}) = \bigcap_{j = 1}^m \CC [x_1^{\pm 1}, \ldots
     x_{j - 1}^{\pm 1}, x_j, y_j, x_{j + 1}^{\pm 1}, \ldots, x_m^{\pm 1},
x_{m+1},\ldots,x_n] \ . \end{equation}

\noindent(b) The upper cluster algebra $\UU  ( \mathfrak{A})$ is defined as
$$\UU  ( \mathfrak{A})=\bigcap_{({\bf x'},B')}\UU_{\bf x'} ( \mathfrak{A})\ ,$$
where the intersection is over all seeds $({\bf x}',B')$ in the mutation equivalence class of $({\bf x},B)$.
\end{definition}

Observe that each cluster algebra is contained in its upper cluster algebra (see \cite{BFZ}).

\subsection{Poisson structures}
\label{se:poissonstructure}
Cluster algebras are closely related to Poisson algebras. In this section we recall some of the related notions and results. 
 
\begin{definition}
Let $k$ be a field of charactieristic $0$. A Poisson algebra is a pair
$(A,\{\cdot,\cdot\})$ consisting of a commutative $k$-algebra $A$ and a bilinear map
$\{\cdot,\cdot\}:A\otimes A\to A$,  satisfying for all $a,b,c\in A$:
\begin{enumerate}
\item skew-symmetry: $\{a,b\}=-\{b,a\}$ 
\item Jacobi identity: $\{a,\{b,c\}\}+\{c,\{a,b\}\}+\{b,\{c,a\}\}=0$,
\item Leibniz rule: $a\{b,c\}=\{a,b\}c+b\{a,c\}$.
\end{enumerate}
\end{definition} 
If there is no room for confusion we will refer to a Poisson algebra $(A,\{\cdot,\cdot\})$ simply as  $A$.
A {\it Poisson Ideal} $\II$ in a Poisson algebra $A$ is an ideal such that $\{\II,A\}\subset \II$, and if $k$ is of characteristic zero, then a Poisson prime ideal is a prime ideal which is also Poisson.

 Gekhtman, Shapiro and
Vainshtein showed in \cite{GSV} that one can associate Poisson structures to
cluster algebras in the following way.   Let $\mathfrak{A} \subset \CC[x_1^{\pm 1}, \ldots,
x_n^{\pm 1}] \subset \mathfrak{F}$ be a cluster algebra. A Poisson structure
$\{\cdot, \cdot\}$ on
$\CC [x_1, \ldots, x_n]$ is called log-canonical if   $\{
x_i,x_j\}=\lambda_{ij} x_ix_j$ with $\lambda_{ij}\in \CC$ for all $1\le i,j\le n$.
 
The Poisson structure can be naturally extended to $\mathfrak{F}$ by using the identity $0=\{f\cdot f^{-1},g\}$  for all $f,g\in\CC [x_1, \ldots, x_n]$.  We thus obtain that  $\{f^{-1},g\}=-f^{-2}\{f,g\}$ for all $f,g\in \mathfrak{F}$. 
We call $\Lambda=\left( \lambda_{ij}\right)_{i,j=1}^n$ the {\it coefficient matrix} of
the Poisson structure. We say that a Poisson structure on  $\mathfrak{F}$ is compatible with
$\mathfrak{A}$ if it is log-canonical with respect to each cluster $(y_1,\ldots,
y_n)$; i.e., it is log canonical on $\CC[y_1, \ldots, y_n]$.

\begin{remark}
\label{re:Class Poisson}
 A  classification of Poisson structures compatible with cluster algebras was obtained by Gekhtman, Shapiro and Vainshtein in \cite[Theorem 1.4]{GSV}.  It is easy to see from their description that if $n$ is even, then the cluster algebra has an admissible Poisson structure of maximal rank.
\end{remark}

We will refer to  the cluster algebra $\AA$ defined by the initial seeed
$({\bf x},B)$ together with the compatible Poisson structure  defined by the coefficient 
matrix $\Lambda$ with respect to the cluster ${\bf x}$ as the {\it Poisson cluster algebra}
defined by the {\it Poisson seed}  $({\bf x},B,\Lambda)$.

It is not obvious under which conditions  a Poisson seed $({\bf x},B,\Lambda)$ would yield a Poisson bracket $\{\cdot,\cdot\}_\Lambda$   on   $\mathfrak{F}$ such that $\{\AA,\AA\}_\Lambda\subset \AA$. We have, however, the following fact. 

\begin{proposition} 
Let   $({\bf x},B,\Lambda)$ be a Poisson seed and $\AA$ the corresponding cluster algebra. Then $\Lambda$ defines a Poisson algebra structure on the upper bound $\UU_{{\bf x},B}(\AA)$ and the upper cluster algebra $\UU(\AA)$.
\end{proposition}

\begin{proof}
Denote as above by $\{\cdot,\cdot\}_\Lambda$ the Poisson bracket on $\mathfrak{F}$ by $\Lambda$.
  Observe  that  the algebras  $\CC[x_1^{\pm 1},\ldots x_{i-1}^{\pm 1}, x_i,y_i, x_{i+1}^{\pm 1}, \ldots, x_n^{\pm 1}]$   are Poisson subalgebras of the Poisson algebra $\CC[x_1^{\pm 1},\ldots x_n^{\pm 1}]$  for each $1\le i\le m$, as $\{x_i,y_i\}_\Lambda=\{x_i,x_i^{-1}P_i\}_\Lambda\in \CC[x_1,\ldots, x_n]$. If $A$ is a Poisson algebra and $\{B_i\subset A:i\in I\}$ is a family of Poisson subalgebras, then $\bigcap_{i\in I} B_i$ is a Poisson algebra, as well. The assertion follows.
\end{proof}

 \subsection{Compatible Pairs and Their Mutation}
\label{se:Compatible Pairs and Mut}
Section \ref{se:Compatible Pairs and Mut} is dedicated to compatible pairs and their mutation.    Compatible pairs yield important examples of Poisson brackets which are compatible with a given cluster algebra structure, and   as we shall see below, they are also integral in defining quantum cluster algebras. Note that our  
 definition is slightly different from the original one in \cite{bz-qclust}. Let, as above, $m\le n$. 
Consider a pair consisting of a skew-symmetrizable $m\times n$-integer matrix
$B$ with rows labeled by the interval $[1,m]=\{1,\ldots, m\}$ and columns  labeled by   $[1,n]$
 together with a skew-symmetrizable $n\times n$-integer matrix  $\Lambda$  with rows and
columns labeled by $[1,n]$.   

\begin{definition}
\label{def:compa pair}
Let $B$ and $\Lambda$ be as above. We say that the pair $(B,\Lambda)$ is
compatible if the coefficients $d_{ij}$ of the $m\times n$-matrix $D=B\cdot \Lambda$ satisfy
$d_{ij}=d_i\delta_{ij}$
for some positive integers $d_i$ ($i\in [1,m]$).  
\end{definition}
This means that $D=B\cdot \Lambda$ is a $m\times n$ matrix where the only
non-zero entries are positive integers on the diagonal of the principal $m\times
m$-submatrix.

The following fact is obvious.

\begin{lemma}
\label{le:full rank}
Let $(B,\Lambda)$ be a compatible pair. Then $B\cdot \Lambda$ has full rank.
\end{lemma}

  Let  $(B,\Lambda)$  be a compatible pair and let $k\in [1,m]$. We define for
$\varepsilon\in \{+1,-1\}$ a $n\times n$ matrix $E_{k,\varepsilon}$ via 
 
 \begin{itemize}
  \item $(E_{k,\varepsilon})_{ij}=\delta_{ij}$ if  $j\ne k$,
  
\item   $(E_{k,\varepsilon})_{ij}= -1$ if   $i=j= k$,
 
\item  $(E_{k,\varepsilon})_{ij}= max(0,-\varepsilon b_{ik})$ if  $i\ne j= k$.

\end{itemize}

 Similarly, we define a $m\times m$ matrix $F_{k,\varepsilon}$ via 
 
  \begin{itemize}
  \item $(F_{k,\varepsilon})_{ij}=\delta_{ij}$   if  $i\ne k$,
  
\item   $(F_{k,\varepsilon})_{ij}= -1$ if   $i=j= k$,
 
\item  $(F_{k,\varepsilon})_{ij}= max(0,\varepsilon b_{kj})$ if  $i= k\ne j$.

\end{itemize}

  We define a new pair $(B_k,\Lambda_k)$ as
  \begin{equation}
  \label{eq:mutation matrix and Poisson}
   B_k=      F^T_{k,\varepsilon}   B E_{k,\varepsilon}^T \ , \quad
\Lambda_k=E_{k,\varepsilon} \Lambda E_{k,\varepsilon}^T\ ,
   \end{equation}
  where $X^T$ denotes the transpose of $X$. We chose this rather non-straightforward way of defining $E_{k\varepsilon}$ and $F_{k,\varepsilon}$ in order to show how our definition relates to that of \cite{bz-qclust}. We will not need it in what follows.  The motivation for the definition is the following fact.
  
  \begin{proposition}\cite[Prop. 3.4]{bz-qclust}
  \label{pr:comp under mutation}
  The pair  $(B_k,\Lambda_k)$ is compatible. Moreover, $\Lambda_k$ is
independent of the choice of the sign $\varepsilon$. 
  \end{proposition}
 
 The following fact is clear.

\begin{corollary}
Let $\AA$ be a cluster algebra given by an initial seed $({\bf x}, B)$ where $B$ is a $m\times n$-matrix. If $(B,\Lambda)$ is a compatible pair, then $\Lambda$ defines a compatible Poisson bracket on $\mathfrak{F}$ and on $\UU(\AA)$.
\end{corollary}

\begin{example}
\label{ex:acyclic}
 If $m=n$ (i.e.~there are no coefficients/frozen variables) and $B$ has full rank, then $(B, \mu B^{-1})$ is a compatible pair for all $\mu\in \ZZ_{> 0}$ such that $\mu B^{-1}$ is an integer matrix. It follows from \cite[Theorem 1.4]{GSV} that in this case all compatible Poisson brackets arise in this way.
\end{example}
\begin{example}
Recall that double Bruhat cells in complex semisimple connected and simply connected algebraic groups have a natural structure of an upper cluster algebra (see \cite{BFZ}). Berenstein and Zelevinsky showed that the standard Poisson structure is given by compatible pairs relative to this upper cluster algebra structure (see \cite[Section 8]{bz-qclust}).  
\end{example}

    \subsection{Quantum Cluster Algebras}
    \label{se: q-cluster algebra spec}
In this section we recall the the definition of a quantum cluster  algebra, introduced by
Berenstein and Zelevinsky in \cite{bz-qclust}. 
We define, for each skew-symmetric $n\times n$-integer matrix $\Lambda$,
the skew-polynomial ring $\CC_\Lambda^t[x_1,\ldots, x_n]$ to be the $\CC[t^{\pm
1}]$-algebra generated by $x_1,\ldots, x_n$ subject to the relations
$$x_ix_j=t^{\lambda_{ij}} x_jx_i\ .$$

Analogously, the quantum torus $H_\Lambda^t=\CC_\Lambda^t[x_1^{\pm 1},\ldots, x_n^{\pm 1} ]$ is defined as the localization of
$\CC_\Lambda^t[x_1,\ldots, x_n]$ at the monoid generated by ${x_1,\ldots, x_n}$,
which is an Ore set. The quantum torus is clearly contained in the skew-field of
fractions $\mathcal{F}_\Lambda$ of $\CC_\Lambda^t[x_1,\ldots, x_n]$, and the
Laurent monomials define a lattice $L\subset  H_\Lambda^t\subset
\mathcal{F}_\Lambda$   isomorphic to $\ZZ^n$. Denote for  each ${\bf e}=(e_1,\ldots e_n)\in \ZZ^n$ by $x^{\bf e}$ the monomial $x_1^{e_1}\ldots x_n^{e_n}$.

We need the notion of a toric frame in order to define the
quantum cluster algebra.
 
 \begin{definition} 
 A toric frame in $ \mathfrak F$ is a mapping $M:\ZZ^n\to  \mathfrak F-\{0\}$ of
the form
 $$M(c)=\phi(X^{\eta(c)})\ ,$$
 where $\phi$ is a  $\QQ(\frac{1}{2})$-algebra automorphism of $\mathfrak F$ and
$\eta: \ZZ^n\to L$ an isomorphism of lattices.
 \end{definition}
  
  Since a toric frame $M$ is determined uniquely by the images of the standard
basis vectors $\phi(X^{\eta(e_1)})$,\ldots, $\phi(X^{\eta(e_n)})$ of $\ZZ^n$, we
can associate to each toric frame a skew commutative $n\times n$-integer matrix
$\Lambda_M$. 
  We can now  define the quantized version of a seed.
  
  \begin{definition}\cite[Definition 4.5]{bz-qclust}
  A quantum seed is a pair $(M,B)$ where
  \begin{itemize}
  \item $M$ is a toric frame in $\mathfrak F$.
  \item $B$ is a   $n\times m$-integer matrix  with rows labeled by $[1,m]$ and
columns  labeled by  $[1,n]$.  
  \item The pair $(B,\Lambda_M)$ is compatible.
  \end{itemize}
  \end{definition}
  Now we define the seed mutation in direction of an exchangeable index $k\in 
[1,m]$. For each $\varepsilon\in \{ 1,-1\}$ we define a mapping $M_k:
\ZZ^n\to  \mathfrak F$ via
  
  $$M_k(c)=\sum_{p=0}^{c_k} \binom{c_k}{p}_{q^{d_k}{2}} M(E_\varepsilon c
+\varepsilon p b^k)\ , \quad M_k(-c)=M_k(c)^{-1}\ ,$$
  where we use the well-known $q$-binomial coefficients (see e.g. \cite[Equation
4.11]{bz-qclust}), and the matrix $E_{k,\varepsilon}$ defined in Section \ref{se:Compatible Pairs and Mut}.  Define $B_k$ to be obtained from $B$ by the standard matrix mutation in
direction $k$, as in Section \ref{se:Cluster Algebras-def}. One obtains the
following fact.
  
  \begin{proposition}
\cite[Prop. 4.7]{bz-qclust}
(a) The map $M_k$ is a toric frame, independent of the choice of sign
$\varepsilon$. 

\noindent(b) The pair $(B_k,\Lambda_{M_k})$ is a quantum seed. 
  \end{proposition}

Now, given an {\it initial quantum seed} $(B, \Lambda_M) $ denote, in a slight
abuse of notation, by $X_1=M(e_1),\ldots, X_r=M(e_r)$, which we refer to as the
{\it cluster variables} associated to the quantum seed $(M,B)$. Here our
nomenclature differs slightly from \cite{bz-qclust}, since there one considers the
coefficients not to be cluster variables. We now define the seed mutation
$$X_k'=M(-e_k+\sum_{b_{ik}>0} b_{ik}e_i)+ M(-e_k-\sum_{b_{ik}<0} b_{ik}e_i)\ .$$

We obtain that $X_k'=M_k(e_k)$ (see \cite[Prop. 4.9]{bz-qclust}).  
We say that two quantum seeds $(M,B)$ and $(M',B')$ are mutation-equivalent if
they can be obtained from one another by a sequence of mutations. Since
mutations are involutive (see \cite[Prop 4.10]{bz-qclust}), the quantum seeds in
$\mathfrak{F}$ can be grouped in equivalence classes, defined by the relation of
mutation equivalence. The quantum cluster algebra generated by a seed 
$(M,B)\subset \mathfrak F$ is the $\CC[t^{\pm 1}]$-subalgebra generated by
the  cluster variables associated to the seeds in an equivalence class.

\begin{remark}
There are definitions of quantum lower bounds, upper bounds and quantum upper cluster algebras (see \cite[Sections 5 and 7]{bz-qclust}),  analogous to the classical case.
\end{remark}

\section{Intersections of Ideals with  Clusters}
\label{se:key propositions}
In the present chapter we consider the intersection  between Poisson ideals in a cluster algebra and individual clusters. Moreover, we prove quantum analogues of the propositions whenever available.

\begin{proposition} 
\label{pr:TPP super toric}
Let ${\bf x}$ be a cluster,  $rank(\Lambda)=n$  and $\II$ be a non-zero Poisson ideal. Then the ideal $\II$ contains a monomial in  $x^m\in \CC[{\bf x}]$. 
\end{proposition}

\begin{proof}
Notice first that $\II_{\bf x}=\II\cap \CC[x_1,\ldots, x_n]\ne 0$. Indeed, let $0\ne f\in \II$. We can express $f$ as a Laurent polynomial in the variables $x_1,\ldots, x_n$; i.e., $f=x_{1}^{-c_1}\ldots x_n^{-c_n} g$ where $c_1,\ldots, c_n\in \ZZ_{\ge 0}$ and $0\ne g\in\CC[x_1,\ldots, x_n]$. Clearly, $g= x_{1}^{c_1}\ldots x_n^{c_n}f \in \II_{\bf x}$. 

We complete the proof  by contradiction. Let $f=\sum_{{\bf w}\in \ZZ^n} c_{\bf w} x^{\bf w}\in\II_{\bf x}$   We assume that $f$ has the smallest number of nonzero summands such that no monomial term $c_{\bf w} x^{\bf w}$ with $c_{\bf w}\neq 0$ is contained in $\II$. It must therefore have at least two monomial terms.

    Assume, as above that $c_{\bf w},c_{{\bf w}'}\ne 0$ and  denote by ${\bf v}$ the difference ${\bf v}={\bf w}-{\bf w}'$. Since $\Lambda$ has full rank,  there exists $i\in [1,n]$ such that $\{x_i,x^{\bf v}\}\ne 0$. Therefore,  $\{x_i, x^{\bf w}\}=cx_ix^{\bf w}\ne dx_ix^{{\bf w}'}\{x_i, x^{{\bf w}'}\}$ for some $c,d\in \CC$. 
Note that $\{x_i,f\}=\sum_{{\bf w}\in \ZZ^n} c_{\bf w} \lambda_{\bf w}x^{\bf w}x_i$ for certain $ \lambda_{\bf w}\in\ZZ$.
Clearly, $cx_if-\{x_i,f\}\in \II$ and 
$$ cx_if-\{x_i,f\}=(c-c)c_{\bf w} x^{\bf w}x_i+(c-d)c_{{\bf w}'} x^{{\bf w}'}x_i+\ldots\ . $$

Hence, $cx_if-\{x_i,f\}\ne 0$ and it has fewer monomial summands than $f$ which contradicts our assumption. Therefore,   $\II$ contains a  monomial.   The proposition is proved.
\end{proof}

 The following fact is an obvious corollary of Proposition \ref{pr:TPP super toric}.
\begin{proposition}
\label{th:intersection}
Let  $\AA$ be the cluster algebra defined by a Poisson  seed $({\bf x}, B,\Lambda)$ with $rank(\Lambda)=n$.  If  $\II\subset \AA$ is a non-zero  Poisson prime ideal, then $\II$ contains a cluster variable $x_i$.
\end{proposition}

 We also have the following  quantum version of Proposition \ref{pr:TPP super toric}   which is proved  analogously to the classical case.
\begin{proposition} 
\label{pr:TPP super toric q}
Let $({\bf x},B,\Lambda)$ be a quantum seed,  $rank(\Lambda)=n$  and $\II$ a  non-zero two-sided ideal. Then the ideal $\II$ contains a monomial in  $x^m\in \CC_\Lambda[{\bf x}]$. 
\end{proposition}

  \begin{remark}
 We do not have a  quantum version of Proposition \ref{th:intersection} because we do not know whether prime ideals in quantum cluster algebras are completely prime.
 \end{remark}

\section{Poisson ideals in acyclic cluster algebras}
\label{se:Poisson and symp geom}
In this section we recall results from our previous paper \cite{ZW tpc}. As the proofs   are rather short, we shall include them for convenience.
Recall e.g.~from \cite{BFZ} that acyclic cluster algebras associated with an acyclic quiver and with trivial coefficients correspond, up to a reordering of the variables   of the acyclic seed, to cluster algebras defined by a seed $({\bf x},B)$ where $B$ is a skew-symmetric $n\times n$-matrix with $b_{ij}>0$ if $i<j$. 
 
Berenstein, Fomin and Zelevinsky proved in \cite{BFZ} that such a  cluster algebra $\AA$ is equal to both its lower and upper bounds. Thus, it is  Noetherian and, if $B$ has full rank, a Poisson algebra with the Poisson brackets given by compatible  pairs  $(B,\Lambda)$ with $\Lambda=\mu B^{-1}$ for certain $\mu\in \ZZ$ (see Example \ref{ex:acyclic}). In order for $B$ to have full rank we have to assume that $n=2k$ is even.   Let $P_i=m_i^+ +m_i^-$ where $m_i^+$ and $m_i^-$ denote the monomial terms in the exchange  polynomial. Then $\{y_i,x_i\}=\mu_1m_i^+ +\mu_2m_i^-$ for some $\mu_1,\mu_2\in \ZZ$.  We, additionally, want to require that $\mu_1\ne \mu_2$. To assure this, we assume that
\begin{equation}
\label{eq:poisson gen}
\sum_{j=1}^n (b^{-1})_{ij}\left(max(b_{ij},0)+min(b_{ij},0)\right)\ne 0\ 
\end{equation}  
for all $i\in [1,n]$. We have the following result.
\begin{theorem}\cite{ZW tpc}
\label{th:acyclic}
Let $\AA$ be an acyclic cluster algebra over $\CC$ with trivial coefficients of even rank $n=2k$, given by a seed $(x_1,\ldots, x_{n}, B)$ where $B$ is a skew-symmetric $n\times n$-integer matrix satisfying  $b_{ij}>0$ if $i<j$ and suppose that  $B$ and $B^{-1}$ satisfy Equation \ref{eq:poisson gen} for each $i\in[1,n]$. Then, the Poisson cluster algebra  defined by a compatible pair $(B,\Lambda)$ where $\Lambda=\mu B^{-1}$  with  $0\ne \mu\in \ZZ$ contains no non-trivial Poisson prime ideals.
\end{theorem}

\begin{proof}
 Suppose that there exists a non-trivial  Poisson prime ideal $\II$. Then, $\II\cap {\bf x}$ is nonempty by Proposition \ref{pr:TPP super toric}, hence $\II\cap {\bf x}=\{x_{i_1},\ldots, x_{i_j}\}$ for some $1\le i_1\le i_2\le \ldots\le i_j\le 2k$.  Additionally, observe that   $P_{i_1}=m_{i_1}^++m_{i_1}^-$ has to be contained in $\II$, as well as 
 $$\{y_{i_1},x_{i_1}\}=\mu_1m_{i_1}^+ +\mu_2m_{i_1}^-\ .$$ 
 By our assumption, we have $\mu_1\ne \mu_2$, and therefore $m_{i_1}^-\in \II$.   Hence, $x_h\in\II$ for some $h\in[1, i_{1}-1]$  or $1\in \II$. But  $b_{i,h}<0$ implies that $h<i$  and we obtain the desired contradiction. The theorem is proved.
\end{proof}

   The theorem has the following corollary which was also independently proved by Muller very recently \cite{Mu 1}, though in more generality.
\begin{corollary}
\label{co:acyclic smooth}
 Let $\AA$ be as in Theorem \ref{th:acyclic}. Then, the variety $X$ defined by $\AA=\CC[X]$ is smooth. 
\end{corollary}   
   \begin{proof}
   The singular subset is contained in  a Poisson ideal of co-dimension greater or equal to one by a result of Polishchuk (\cite{Pol}). It is well known that  the Poisson  ideal must be contained in a proper Poisson prime ideal (see also \cite{ZW tpc}). 
   The assertion follows.
   \end{proof}
   
\begin{remark}
The assumption that the cluster algebra has even rank is very important. Indeed, Muller has recently shown that the variety corresponding to the cluster algebra of type $A_3$ has a singularity (\cite[Section 6.2]{Mu}).
\end{remark}

\subsection{Symplectic Structure}
\subsubsection{Symplectic geometry of Poisson varieties}
In this section we will recall some well-known properties  of  the symplectic structure on Poisson varieties. Our discussion follows along the lines of \cite[Part III.5]{brown-goodearl}. If $A$ is a Poisson algebra over a field $k$, then each $a\in A$ defines a derivation $X_a$ on $A$ via
$$X_a(b)=\{a,b\}\ .$$
This derivation is called the {\it Hamiltonian vectorfield} of $a$ on $A$.

Now suppose that $A$ is the coordinate ring of an affine complex variety $Y$.
We will associate to the Poisson bracket the {\it Poisson bivector} $u\in \Lambda^2 T(Y)$ where $T(Y)$ denotes the tangent bundle of $Y$. 
Let $p\in Y$ be a point and $\mm_p\subset A$ the corresponding maximal ideal. Let $\alpha,\beta\in \mm_p/\mm_p^2$ be elements of the cotangent space and let $f,g$ be lifts of $\alpha$ and $\beta$, respectively. We define $u_p\in \Lambda^2 T_p(Y)$
$$u_p(\alpha,\beta)=\{f,g\}(p)\ .$$   
Note that $u_p$ is a well-defined skew-symmetric form. Indeed, if $I\subset A$ is an ideal and $b\in \II^2$, then for all $a\in A$
$$ \{a,b\}\in \II\ $$
by the Leibniz rule.  The form $u_P$ may be degenerate, indeed if it is non-degenerate at every point $p\in Y$, then we call $u_P$ symplectic and, moreover, if  $Y$ is connected, then $Y$ is smooth and a (holomorphic) symplectic manifold.   Define 
$$N(p)=\{\alpha \in T^*_p Y:u_p(\alpha,\cdot)=0\}\ ,$$
and $H(p)\subset T_p Y$ its orthogonal complement.  The space $H(P)$ is the tangent space  of the linear span of the Hamiltonian vectorfields at $p$. Recall that by the Theorem of Frobenius, a Poisson variety $Y$ decomposes as a disjoint union of symplectic leaves, maximal symplectic submanifolds. The tangent space of the symplectic leaf at the point $p$ is  $H(p)$.  
 \subsubsection{The Main Theorem}
 Corollary \ref{co:acyclic smooth} implies that $X$ is smooth, hence it has the structure of  a complex manifold.  We  have the following result. 
 \begin{theorem}
 \label{th:symplectic leaves}
 Let $\AA$ be an acyclic cluster algebra as defined above, $X$ an affine variety such that $\AA=\CC[X]$ and let $\Lambda$ define a compatible Poisson bracket. Then, $X$ is a holomorphic symplectic manifold.
 \end{theorem}
 \begin{proof}
 First, let $p\in X$ be a  generic point, by which we mean that $x_i(p)\ne 0$ for all $i=[1,\ldots,n]$. Set $x_i(p)=p_i$. It is easy to see that the Hamiltonian vectorfield
 $X_{x_i}$ at $p$ evaluates in the local coordinates $(x_1,\ldots, x_n)$ as
$$ X_{x_i}(p)= (\lambda_{1i} p_i,\ldots, \lambda_{ni} p_i)\ ,$$
where $\Lambda=(\lambda_{ij})_{i,j=1}^n$.

Since, $\Lambda$ is non-degenerate, we obtain that the Hamiltonian vectorfields span the tangent space $T_p X$ at $p$. It remains to consider the case when $p_i=0$ for some $i\in[1,n]$. Suppose that $p\in X$ such that $p_i=0$ and $p_j\ne 0$ for all $j<i$.  We have to show that the symplectic leaf containing $p$ is not contained in the hyper-surface $x_i=0$. We may assume, employing  induction,  that if $p_1,\ldots, p_i\ne 0$ then the symplectic leaf at $p$ has full rank.  We now claim that 
$\{x_i,y_i\}(p)\ne 0 $. Indeed, suppose that $\{x_i,y_i\}(p)=(\mu_1 m^+\mu_2 m^-)(p)=0$. Since $p_i=0$ implies that $(m^++m^-)(p)=0$, we would conclude as in the proof of Theorem \ref{th:acyclic} that $m^+(p)=m^-(p)=0$, but that is a contradiction to our assumption that $p_j\ne 0$ for all $j<i$.  Denote by $u\in  \Lambda^2 T(Y)$ the Poisson bivector.  We obtain that  $u_p(\frac{\delta}{\delta x_i},\cdot)\ne 0$, hence the symplectic leaf containing $p$ is not tangent to the hypersurface $x_i=0$ at  $p$.  It must contain  a point in an analytic  neighborhood of $p$ at which $x_i(p)\ne 0$ and $x_j(p)\ne 0$ for all $j<i$ by our assumption. We obtain the desired contradiction and, hence,  every symplectic leaf has dimension $n$. But the manifold $X$ is connected and, therefore,  cannot be a union of disjoint open submanifolds of equal dimension, hence $X$  contains only one symplectic leaf and the  theorem is proved.    

 \end{proof}
 \begin{remark}
 This result can be easily generalized to acyclic cluster algebras with invertible coefficients (using Remark \ref{re:Class Poisson}). This would imply that in the set-up of locally acyclic cluster algebras it should be easy to show that the spectrum of a locally acyclic cluster algebra is a holomorphic symplectic manifold. 
 \end{remark}
 
 \section{Ideals in Acyclic Quantum Cluster Algebras}
 \label{se:ideals in qca}
 \begin{theorem}
 \label{th:ideals in qclalg}
 Let $\AA_q$ be a quantum cluster algebra with quantum seed $({\bf x}, B,\Lambda)$ satisfying Equation \ref{eq:poisson gen}. Then, $\AA_q$ does not contain any non-trivial proper prime ideals.
 \end{theorem}
 
 \begin{proof} Let $\II$ be a prime ideal in $\AA_q$. 
 We obtain from Proposition \ref{pr:TPP super toric q} that  $\II$ contains a monomial $x^{\bf v}$ with ${\bf v}\in\ZZ_{\ge 0}^n$.   It is easy to observe that we can choose  ${\bf v}$ to be  minimal with respect to the lexicographic order on $\ZZ^n_{\ge 0}$. Recall that the lexicographic ordering defines   ${\bf u}<{\bf w}$ for  ${\bf u},{\bf w}\in \ZZ^n$ if and only if there exists $i\in [1,n]$ such that $u_i,w_i$ and $u_j=w_j$ if $j>i$. Recall that this defines a total ordering on $\ZZ^n_{\ge0}$.  There exists some $i\in[1,n]$ such that $v_i\ne 0$ and $v_k=0$ for all $k>i$, and we write $x^{\bf v}=x^{\bf v'} x_i^{v_i}$ with ${\bf v'}={\bf v}-v_i{\bf  e}_i$. 
 
 Recall that an ideal $\II$ in a non-commutative ring $R$ is prime if $arb\in \II$  for all $r\in R$ implies that $a\in \II$ or $b\in \II$. Now, since $\AA_q$ is acyclic, we know from \cite[Theorem 7.3 and 7.5]{bz-qclust} that it is isomorphic to its lower bound. Hence, employing the notation of Definition \ref{def:lower bounds}, each element  $a\in \AA_q$ can be written as  $a=\sum_{p=1}^r c_p x^{\bf w}y_i^h y^{\bf w'}$ where $w_kw'_k=w_i h =0$ and $c_p\in\CC[q^{\pm 1}]$ for all $p\in[1,r]$. Observe that $y_k$ and $x_\ell$ skew-commute if $k\ne \ell$, as they are cluster variables in the cluster ${\bf x}_k$. We now compute that  
 $$  x^{\bf v'} ax_i^{j}=\sum_{p=1}^r  c_p q^{\lambda_a} x^{\bf w} y_i^h x^{\bf v'} x_i^{v_i} y^{w'}= \sum_{p=1}^r  c_p q^{\lambda_p} x^{\bf w} y_i^h  x^{\bf v}  y^{w'} \in \II\ ,$$
 for all $a\in\AA_q$ and certain $\lambda_p\in\ZZ$. Hence, since $\II$ is a prime ideal, $x^{\bf v'}\in \II$ or $x_{i}^{v_i}\in\II$. But ${\bf v'}<{\bf v}$ and $v_i{\bf e}_i\le {\bf v}$, and therefore,   $ x_i^{v_i}=x^{\bf v}\in\II$.  But we have assumed  in Equation \ref{eq:poisson gen} that 
 $x_iy_i$ and $y_ix_i$, are not linearly dependent, hence $x_i^{v_i} y_i$ and $y_i x_i^{v_i}$ are not linearly dependent. As in the proof of Theorem \ref{th:acyclic}, we now argue that    $x_i^{v_i-1 }m_i^{+}\in \II\ $, where $P_i=m_i^++m_i^-$
  (see  Section \ref{se:Poisson and symp geom}).  But this is a contradiction to our minimality assumption. The theorem is proved. 
 \end{proof}
 
 Theorem  \ref{th:ideals in qclalg} and Theorem \ref{th:symplectic leaves} have the following immediate corollary.
 
 \begin{corollary}
 Let $({\bf x}, B,\Lambda)$ be an acyclic Poisson or quantum seed, as defined above. Then 
 the space of primitive ideals in the quantum cluster algebra (one point corresponding to the $0$ ideal) and the space of symplectic leaves (also just one point) are homeomorphic.
 \end{corollary}

\end{document}